\documentclass[11pt,reqno]{amsart}
\usepackage{}
\usepackage{bbm}
\usepackage{url}
\usepackage{stmaryrd}
\usepackage{mathrsfs}
\usepackage{cases}
\usepackage{amsfonts}
\usepackage{amssymb}
\usepackage{amsmath}
\usepackage{tikz}
\usepackage{extarrows}
\usepackage{enumerate}
\allowdisplaybreaks[4]
\newtheorem{theorem}{Theorem}[section]
\newtheorem{lemma}[theorem]{Lemma}

\newtheorem{corollary}[theorem]{Corollary}
\theoremstyle{definition}
\newtheorem{definition}[theorem]{Definition}

\numberwithin{equation}{section}

\def\ls{\lesssim}

\let\b=\beta
\let\g=\gamma
\let\d=\delta

\let\la=\lambda
\let\r=\rho

\let\om=\omega

\let\Om=\Omega

\let\ep=\epsilon
\let\va=\varphi

\def\bbR{\mathbb{R}}
\def\bbS{\mathbb{S}}

\def\calO{\mathcal {O}}
\def\tcalO{\widetilde{\calO}}
\def\calV{\mathcal {V}}
\def\calT{\mathcal {T}}

\newcommand{\be}{\begin{equation*}}
\newcommand{\ee}{\end{equation*}}
\newcommand{\ben}{\begin{equation}}
\newcommand{\een}{\end{equation}}
\newcommand{\bn}{\begin{enumerate}}
\newcommand{\en}{\end{enumerate}}
\newcommand{\ba}{\begin{align*}}
\newcommand{\ea}{\end{align*}}
\newcommand{\ban}{\begin{align}}
\newcommand{\ean}{\end{align}}
\newcommand{\bs}{\backslash}

\def\apn{{A_p}}
\def\bmo{{{\rm BMO}(\mathbb R^n)}}
\def\bmoz{{{\rm BMO}_\lambda(\mathbb R^n)}}
\def\cmo{{{\rm CMO}(\mathbb R^n)}}

\def\rn{{{\bbR}^n}}

\begin{document}
\title[On the compactness of oscillation and variation of commutators]
{On the compactness of oscillation and variation of commutators}
\author{WEICHAO GUO}
\address{School of Science, Jimei University, Xiamen, 361021, P.R.China}
\email{weichaoguomath@gmail.com}
\author{YONGMING WEN}
\address{School of Mathematical Sciences, Xiamen University,
Xiamen, 361005, P.R. China} \email{wenyongmingxmu@163.com}
\author{HUOXIONG WU}
\address{School of Mathematical Sciences, Xiamen University,
Xiamen, 361005, P.R. China} \email{huoxwu@xmu.edu.cn}
\author{DONGYONG YANG}
\address{School of Mathematical Sciences, Xiamen University,
Xiamen, 361005, P.R. China} \email{dyyang@xmu.edu.cn}

\begin{abstract}
In this paper, we first establish the weighted compactness result for oscillation and variation associated with
the truncated commutator of singular integral operators. Moreover, we establish a new $\cmo$
characterization via the compactness of oscillation and variation of commutators on weighted Lebesgue spaces.
\end{abstract}
\thanks{Supported by the NSF of China (Nos.11771358, 11471041, 11701112, 11671414,11871254), and the NSF of Fujian Province of China (Nos.2015J01025, 2017J01011).}

\maketitle

\section{Introduction}\label{s1}
The singular integral operator with homogeneous kernel is defined by
\begin{equation}\label{integral operator, linear case}
  T_{\Omega}f(x):=\text{p.v.}\int_{\mathbb{R}^n}\frac{\Omega(x-y)}{|x-y|^{n}}f(y)dy,
\end{equation}
where $\Omega$ is a homogeneous function of degree zero and satisfies the following mean value zero property:
\begin{equation}\label{mean value zero}
  \int_{\mathbb{S}^{n-1}}\Omega(x')d\sigma(x')=0,
\end{equation}
where $d\sigma$ is the spherical measure on the sphere $\mathbb S^{n-1}$.
Given a locally integrable function $b$ and a linear operator $T$, the commutator $[b,T]$
  is defined by
  \be
  T^b(f)(x):=[b,T]f(x):=b(x)T(f)(x)-T(bf)(x)
  \ee
 for suitable functions $f$.
The famous work of Coifman, Rochberg and Weiss \cite{CoifmanRochbergWeiss76AnnMath} gave a characterization of $L^p$-boundedness
of $[b,R_j]$, for every Riesz transform $R_j$.
This result was improved by Uchiyama in his remarkable work \cite{Uchiyama78TohokuMathJ},
in which he showed that the commutator $[b,T_{\Omega}]$ with $\Om\in Lip_{1}(\bbS^{n-1})$
is bounded (compact resp.) on $L^p(\rn)$ if and only if
the symbol $b$ is in $\bmo$ ($\cmo$ resp.),
where $\cmo$ denotes the closure of $C_c^{\infty}(\mathbb{R}^n)$ in the ${\rm BMO}(\mathbb{R}^n)$ topology.
Since then, the work on compactness of commutators of singular and fractional integral operators and its applications to PDE's
have been paid more and more attention; see, for example,
\cite{ChaffeeTorres15PA,ChenHu15CMB, ChenDingWang09PA,ClopCruz13AASFM, Iwaniec92,KrantzLi01JMAAb, Taylor11}
and the references therein.
Recently, inspired by
Lerner--Ombrosi--Rivera-R\'ios \cite{LernerOmbrosiRivera-RiosBLMSta},
the first, third and fourth authors \cite{GuoWuYang17Arxiv}
give some new characterizations of the compact commutators of singular integrals via
${\rm CMO}(\mathbb{R}^n)$.

This paper is devoted to a first contribution to the weighted $L^p$-compactness of the
oscillation and variation of the commutator of singular integral operator.
To state our main results, we first recall some notations.

For a one-parameter family of operators
$\mathcal {W}:=\{W_t\}_{t>0}$, the variation of $\mathcal {W}$ is defined by
\be
\calV_{\r}(\mathcal {W})(f)(x):=\sup_{\ep_{i}\downarrow 0}\left(\sum_{i=1}^{\infty}|W_{\ep_{i+1}}f(x)-W_{\ep_{i}}f(x)|^{\r}\right)^{1/\r},
\ \ \ \ \ (\r>2).
\ee
In general, the boundedness estimate of variation operators can fail when $\r\leq 2$, see the case of martingales in \cite{Qian1998AnnofProb}.

Next, we recall the definition of the oscillation operator of $\mathcal {W}$:
\be
\calO(\mathcal {W}f)(x):=\left(\sum_{i=1}^{\infty}\sup_{t_{i+1}\leq \ep_{i+1}<\ep_i\leq t_i}|W_{\ep_{i+1}}f(x)-W_{\ep_{i}}f(x)|^2\right)^{1/2},
\ee
where $\{t_j\}$ is a decreasing sequence of positive numbers converging to $0$.

The variation inequality was first proved by L\'{e}pingle \cite{Lepingle1976ZWuVG}
for martingales. Then, Bourgain \cite{Bourgain1989IHSPM} proved the variation inequality
for the ergodic averages of a dynamic system.
Since then,
the oscillation and variation have been the active subject of recent research in the field of probability,
ergodic theory and harmonic analysis.
In 2000, Campbell et al. \cite{CampbellJonesReinholdWierdl2000DMJ}
established the $L^p(\rn)$-boundedness of variation for truncated Hilbert transform
and then extended to higher dimensional case in \cite{CampbellJonesReinholdWierdl2003TAMS}.
For the weighted boundedness result one can see \cite{GillespieTorrea2004IJM}, \cite{MaTorreaXu2015JFA}
and \cite{MaTorreaXu2017SCM}.

We say that $T_{K}$ is a Calder\'{o}n-Zygmund operator on $\bbR^n$
if $T_{K}$ is bounded on $L^2(\mathbb{R}^n)$ and it admits the following
representation
\ben\label{representation of singular integral operator}
T_{K}f(x)=\int_{\bbR^n}K(x,y)f(y)dy\ \ \ \text{for all}\ x\notin \text{supp}f
\een
with kernel $K$ satisfying the size condition
\ben\label{conditon of kerenel 1}
|K(x,y)|\leq \frac{C_{K}}{|x-y|^n}
\een
and a smoothness condition
\ben\label{conditon of kerenel 2}
|K(x,y)-K(x',y)|+|K(y,x)-K(y,x')|\leq \frac{C_K}{|x-y|^n}\left(\frac{|x-x'|}{|x-y|}\right)^{\g},
\een
for all $|x-y|>2|x-x'|$, where $C_K>0$, $\g>0$.

\begin{definition}\label{d-bmo}
 The space of functions with bounded mean oscillation, denoted by ${\rm BMO}(\mathbb{R}^n)$, consists of all
$f\in L_{loc}^1(\mathbb{R}^n)$ such that
\be
\|f\|_{\bmo}:=\sup_{Q\subset \mathbb{R}^n}\mathcal{O}(f;Q)<\infty,
\ee
where
\be
f_Q:=\frac{1}{|Q|}\int_{Q}f(y)dy\,\,{\rm and}
\,\,
\mathcal{O}(f;Q):=\frac{1}{|Q|}\int_Q |f(x)-f_Q|dx.
\ee
\end{definition}

The following class of $\apn$ was introduced by Muckenhoupt \cite{Muckenhoupt72TAMS} to
study the weighted norm inequalities of Hardy-Littlewood maximal operators.

\begin{definition}\label{d-Ap weight}
For $1<p<\infty$, the Muckenhoupt class $\apn$ is the set of non-negative locally integrable functions $\om$ such that
\be
[\om]_{A_p}^{1/p}:=\sup_{Q}\left(\frac{1}{|Q|}\int_{Q}\om(x)dx\right)^{1/p}\left(\frac{1}{|Q|}\int_{Q}\om(x)^{1-p'}dx\right)^{1/p'}<\infty,
\ee
where $1/p+1/p'=1$.
\end{definition}

Our main results can be formulated as follows.
\begin{theorem}\label{theorem, compactness}
  Let $1<p<\infty$, $\om\in \apn$ and $b\in \cmo$. We have the following two statements.
  \bn
  \item The $L^p_{\om}(\rn)$-boundedness of $\calO(\calT_{K})$ implies the $L^p_{\om}(\rn)$-compactness of $\calO(\calT_{K}^b)$;
  \item The $L^p_{\om}(\rn)$-boundedness of $\calV_{\r}(\calT_{K})$ implies the $L^p_{\om}(\rn)$-compactness of $\calV_{\r}(\calT_{K}^b)$.
  \en
\end{theorem}

In order to establish the necessity and
equivalent characterization of compact oscillation operator, we define
the modified oscillation by
\be
\tcalO(\mathcal {W}f)(x):=\left(\sum_{i=1}^{\infty}\sup_{t_{i+1}\leq \ep_{i+1}<\ep_i\leq t_i}|W_{\ep_{i+1}}f(x)-W_{\ep_{i}}f(x)|^2\right)^{1/2}
+|W_{t_1}f(x)|.
\ee
This variant of oscillation is necessary for the following Theorem \ref{theorem, necessity of cp}
and Corollary \ref{corollary, characterization-cp}, since one can choose a function $b\notin \bmo$ such
that $\calO(\calT_{\Omega}^b)$ in Theorem \ref{theorem, necessity of cp} and Corollary \ref{corollary, characterization-cp}
is a compact operator on $L^p_{\om}(\rn)$.
We put the details in Appendix A.

\begin{theorem}\label{theorem, necessity of cp}
  Let $1<p<\infty$, $b\in L_{loc}^1(\mathbb{R}^n)$ and $\om\in \apn$.
  Let $\Om$ be a bounded measurable function on $\bbS^{n-1}$, which does not change sign and is not equivalent to zero
  on some open subset of  $\bbS^{n-1}$.
  Then we have the following two statements.
  \bn
  \item Let $\{t_j\}_{j=1}^{\infty}$ be a sequence with $\sup_{i\in \mathbb{Z}}|\{j: 2^i\leq |t_j|<2^{i+1}\}|<\infty$.
  Then the $L^p_{\om}(\rn)$-compactness of
  $\tcalO(\calT_{\Om}^b)$ implies $b\in \cmo$;
  \item The $L^p_{\om}(\rn)$-compactness of
  $\calV_{\r}(\calT_{\Om}^b)$ implies $b\in \cmo$.
  \en
\end{theorem}

\begin{corollary}\label{corollary, characterization-cp}
  Let $1<p<\infty$, $b\in L_{loc}^1(\mathbb{R}^n)$, $\om\in \apn$ and $\Om\in Lip_1(\bbS^{n-1})$ with $\Omega\not\equiv0$. Then
  \bn
  \item $b\in \cmo\Longleftrightarrow \tcalO(\calT_{\Om}^b)$ is compact on $L^p_{\om}(\rn)$;
  \item $b\in \cmo\Longleftrightarrow \calV_{\r}(\calT_{\Om}^b)$ is compact on $L^p_{\om}(\rn)$.
  \en
\end{corollary}

This paper is organized as follows.
Section 2 is devoted to the proof of the sufficiency of compactness, i.e., Theorem \ref{theorem, compactness}.
It is well known that the Fr\'echet-Kolmogorov theorem
is a powerful tool in the study of compactness of commutators of singular integral operators,
see, for example, \cite{Uchiyama78TohokuMathJ}.
In the proof of Theorem \ref{theorem, compactness}, we also
use the weighted Fr\'echet-Kolmogorov theorem (see Lemma \ref{lemma, criterion of precompact})
to prove the compactness of $\calO(\calT_{K}^b)$ and $\calV_{\r}(\calT_{K}^b)$.
However, due to the special structures of oscillation and variation,
the argument here is more complicated. 
Moreover, compared to the known case of singular integral operators,
the regularity of commutator of oscillation or variation of a singular integral operator
comes from not only the regularity of symbol $b$ and the kernel $K$,
but also the smallness of corresponding measurable sets degenerated by the annuluses in the definition of oscillation or variation. 

The necessity conditions of compactness will be dealt with in Section 3.
By establishing two claims A and B, we reduce our cases to the known cases in \cite{GuoWuYang17Arxiv}.
Then, Theorem \ref{theorem, necessity of cp} can be proved.
Appendix A is used to clarify the reasonableness  of the modified oscillation
in our results for the necessity.

We remark that all conclusions of this article can de extended to the high order commutator case (oscillation and variation of high order commutators) as in \cite{GuoWuYang17Arxiv}.
We omit such more complicated expression form just for concise, and leave the details for the interested readers.

Throughout this paper, we will adopt the following notations. Let $C$ be a
positive constant which is independent of the main parameters.
The notation $X\lesssim Y$ denotes the
statement that $X\leq CY$, $X\sim Y$ means $X\lesssim Y\lesssim X$. For a
given cube $Q$, we use $c_Q$, $l_Q$ and $\chi_Q$
to denote the center, side length and characteristic function of $Q$. We also denote $(A\setminus B)\cup (B\setminus A)$ by $A\triangle B$.
For any point $x_0\in\rn$ and sets $E, F\subset\rn$,
$E+x_0:=\{y+x_0: y\in E\}$ and $E-F:=\{x-y: x\in E, y\in F\}$.

\section{Compactness of the oscillation operators}
In this part, we study the compactness property of oscillation and variation.
Thanks to the so called conjugation method (see, for example, \cite{PerezRivera16arxiv}) and John-Nirenberg inequality, the boundedness of
$\calO(\calT_{K}^b)$ and $\calV_{\r}(\calT_{K}^b)$ can be deduced by the weighted boundedness of
$\calO(\calT_{K})$ and $\calV_{\r}(\calT_{K})$ respectively. Precisely, we have the following lemma.

\begin{lemma}\label{lemma, bd}
  Let $1<p<\infty$, $\om\in \apn$ and $b\in \bmo$. We have the following two statements.
  \bn
  \item The $L^p_{\om}(\rn)$-boundedness of $\calV_{\r}(\calT_{K})$ implies the $L^p_{\om}(\rn)$-boundedness of $\calV_{\r}(\calT_{K}^b)$
  with
  $$\|\calV_{\r}(\calT_{K}^b)\|_{L^p_{\om}(\rn)\rightarrow L^p_{\om}(\rn)}\lesssim \|b\|_{\bmo}.$$
  \item The $L^p_{\om}(\rn)$-boundedness of $\calO(\calT_{K})$ implies the $L^p_{\om}(\rn)$-boundedness of $\calO(\calT_{K}^b)$ with
  $$\|\calO(\calT_{K}^b)\|_{L^p_{\om}(\rn)\rightarrow L^p_{\om}(\rn)}\lesssim \|b\|_{\bmo}.$$
 \en
\end{lemma}
The conclusion (1) of Lemma \ref{lemma, bd} was proved in \cite[Theorem 1.1]{ChenDingHongLiu2017Arxiv}.
Since the proof of (2) in Lemma \ref{lemma, bd} is similar, we omit the details here.

Now, we turn to the proof of Theorem \ref{theorem, compactness}.
We recall the weighted Fr\'echet-Kolmogorov theorem \cite{ClopCruz13AASFM} as follows.
\begin{lemma}\label{lemma, criterion of precompact}
  Let $p\in (1,\infty)$ and $\om\in A_p$. A subset $E$ of $L^p_{\om}(\rn)$ is precompact (or totally bounded)
  if the following statements hold:
  \bn[(a)]
  \item $E$ is bounded, i.e., $sup_{f\in E}\|f\|_{L_{\omega}^p(\mathbb{R}^n)}\lesssim 1$;
  \item $E$ uniformly vanishes at infinity, that is,
  \be
  \lim_{N\rightarrow \infty}\int_{|x|>N}|f(x)|^p\om(x)dx=0,\ \text{uniformly for all}\ f\in E.
  \ee
  \item $E$ is uniformly equicontinuous, that is,
  \be
  \lim_{\r\rightarrow 0}\sup_{y\in B(0,\r)}\int_{\bbR^n}|f(x+y)-f(x)|^p\om(x)dx=0,\ \text{uniformly for all}\ f\in E.
  \ee
  \en
\end{lemma}
Then, we collect some basic properties of the Muckenhoupt class $\apn$.
One can see \cite{Grafakos08} for the proofs of (i)-(iii) of Lemma \ref{lemma, property of weight}.
\begin{lemma}\label{lemma, property of weight}
Let $1<p<\infty$.
\begin{itemize}
  \item [{\rm(i)}] $\om\in A_p\Longleftrightarrow \om^{1-p'}=\om^{-p'/p}\in A_{p'}$.
  \item [{\rm (ii)}] If $\om\in A_p$, there exists a small constant $\ep$ depending only on $n$, $p$ and $[\om]_{A_p}$ such that
  $\om\in A_{p-\ep}$.
  \item [{\rm (iii)}]For all $\la>1$, and all cubes $Q$,
\be
\om(\la Q)\leq \la^{np}[\om]_{A_p}\om(Q).
\ee
  \item [{\rm (iv)}] If $\om\in A_p$, we have
  \be
  \lim_{N\rightarrow +\infty}\int_{B(0,N)^c}\frac{\om(x)}{|x|^{np}}dx=0,\ \ \
  \lim_{N\rightarrow +\infty}\int_{B(0,N)^c}\frac{\om^{1-p'}(x)}{|x|^{np'}}dx=0.
  \ee
\end{itemize}
\end{lemma}
\begin{proof}
We only give the proof of (iv). Since $\om\in A_p$, there exists $\ep>0$ such that $\om \in A_{p-\ep}$.
Write
\be
\begin{split}
\int_{B(0,N)^c}\frac{\om(x)}{|x|^{np}}dx
= &
\sum_{j=0}^{\infty}\int_{2^jN\leq |x|< 2^{j+1}N}\frac{\om(x)}{|x|^{np}}dx
\\
\lesssim &
\sum_{j=0}^{\infty}(2^jN)^{-np}\om(B(0,2^{j+1}N))
\\
\lesssim &
\sum_{j=0}^{\infty}(2^jN)^{-np}(2^jN)^{n(p-\ep)}
=
N^{-n\ep}\sum_{j=0}^{\infty}2^{-jn\ep}\longrightarrow 0
\end{split}
\ee
as $N\rightarrow +\infty$,
where we use property (iii) in the second inequality.
Similarly, by property (i), we get $\om^{1-p'}\in A_{p'}$, then the second equality of (iv) follows.
\end{proof}

We now present the proof of Theorem \ref{theorem, compactness}. We point out that
since there is some essential difference between the arguments for oscillation and variation,
we give the proofs of (1) and (2), respectively.

\begin{proof}[Proof of (1) in Theorem \ref{theorem, compactness}]
Assume that $\calO(\calT_{K})$ is bounded on $L^p_{\om}(\rn)$ and $b\in \cmo$.
  Using Lemma \ref{lemma, bd} (1), we see that $\calO(\calT_{K}^b)$
  is also bounded on $L^p_{\om}(\rn)$.  Moreover, by the definition of $\cmo$,
 it suffices to show  $\calO(\calT_{K}^b)$ is compact on $L^p_{\om}(\rn)$
  for $b\in C_c^{\infty}(\bbR^n)$.
  To this end, we follow the idea in \cite{KrantzLi01JMAAb} and
  consider smooth truncated singular integral operators.
  Take $\va\in C_c^{\infty}(\bbR^n)$ supported on $B(0,1)$ such that $\va=1$ on $B(0,1/2)$, $0\leq \va\leq 1$.
Let $\delta>0$ be a small constant,
$$\va_{\d}(x):=\va(\frac{x}{\d}),\,\, K^{\d}(x,y): =K(x,y)\cdot (1-\va_{\d}(x-y)),$$
and
$$T_{K^{\d}}^b(f)(x):=b(x)\int_{\bbR^n}K^{\d}(x,y)f(y)dy-\int_{\bbR^n}K^{\d}(x,y)b(y)f(y)dy.$$

We first claim that for any $f\in L^p_{\om}(\rn)$,
\begin{equation}\label{pf 1,1}
\|\calO(\calT_{K}^b)(f)-\calO(\calT_{K^{\d}}^b)(f)\|_{L^p_{\om}(\rn)}\lesssim \d\|f\|_{L^p_{\om}(\rn)},
\end{equation}
where the implicit constant is independent of $f$. In fact, A simple calculation yields that the kernel $K^{\d}$ also satisfies
\eqref{conditon of kerenel 1} and \eqref{conditon of kerenel 2} with $C_K$
replaced by certain constant $C$ therein.
Moreover, $K^{\d}$ is a bounded function since for any $x, y$ with $x\not=y$,
\ben\label{pf bd of Kd}
|K^{\d}(x,y)|
\lesssim \frac{1}{|x-y|^n}(1-\va_{\d}(x-y))
\leq \frac{1}{|x-y|^n}\chi_{\{(x,y):|x-y|\geq \d/2\}}\lesssim \frac{1}{\d^n}.
\een
By the sub-linearity of oscillation, we have
\be
|\calO(\calT_{K}^b)(f)-\calO(\calT_{K^\d}^b)(f)|\leq \calO(\calT_{K}^b-\calT_{K^{\d}}^b)(f).
\ee
From this and the fact that for any $x$ and $y$,
$$|b(x)-b(y)|\leq \|\nabla b\|_{L^{\infty}(\rn)}|x-y|,\,\, \textrm{and}\,\, |\va_{\d}(x-y)|\leq \chi_{\{(x,y):|x-y|\leq \d\}}(x,y),$$
we further deduce that for any $x$,
\begin{eqnarray*}
  &&|\calO(\calT_{K}^b)(f)(x)-\calO(\calT_{K^{\d}}^b)(f)(x)|\nonumber
  \\
  &&\quad\leq
  \left(\sum_{i=1}^{\infty}\sup_{t_{i+1}\leq \ep_{i+1}<\ep_i\leq t_i}\left|\int_{\ep_{i+1}<|x-y|\leq \ep_i}(b(x)-b(y))\va_{\d}(x-y)K(x,y)f(y)dy\right|^2\right)^{1/2}\nonumber
  \\
  &&\quad\lesssim
  \left(\sum_{i=1}^{\infty}\sup_{t_{i+1}\leq \ep_{i+1}<\ep_i\leq t_i}\left|\int_{\ep_{i+1}<|x-y|\leq \ep_i}\va_{\d}(x-y)\frac{|f(y)|}{|x-y|^{n-1}}dy\right|^2\right)^{1/2}
  \\
  &&\quad\leq
  \left(\sum_{i=1}^{\infty}\left|\int_{t_{i+1}<|x-y|\leq t_i}\va_{\d}(x-y)\frac{|f(y)|}{|x-y|^{n-1}}dy\right|^2\right)^{1/2}\nonumber
  \\
  &&\quad\leq
  \sum_{i=1}^{\infty}\int_{t_{i+1}<|x-y|\leq t_i}\va_{\d}(x-y)\frac{|f(y)|}{|x-y|^{n-1}}dy
  \leq
  \int_{B(x,\d)}\frac{|f(y)|}{|x-y|^{n-1}}dy\lesssim \d M(f)(x),\nonumber
\end{eqnarray*}
where $M(f)$ is the Hardy-littlewood maximal function of $f$, and the implicit constant is independent of $x,\d$ and $f$.
This via the boundedness of $M(f)$ on $L^p_{\om}(\rn)$ implies that
\begin{equation}\label{e-osci commu donimate HL maxi}
\|\calO(\calT_{K}^b)(f)-\calO(\calT_{K^{\d}}^b)(f)\|_{L^p_{\om}(\rn)}\lesssim \d \|M(f)\|_{L^p_{\om}(\rn)}\lesssim \d\|f\|_{L^p_{\om}(\rn)}
\end{equation}
and shows the claim \eqref{pf 1,1}.

Now observe that to show $\calO(\calT_{K}^b)$ is compact on $L^p_{\om}(\rn)$,
we only need to show that
$$\{\calO(\calT_{K}^b)(f): \|f\|_{L_{\omega}^p(\mathbb{R}^n)}\leq 1\}$$
is precompact. Then by \eqref{pf 1,1},
it suffices to show that for $\d$ small enough,
the set
$$A(\calO(\calT_{K^{\d}}^b)):=\{\calO(\calT_{K^{\d}}^b)(f): \|f\|_{L_{\omega}^p(\mathbb{R}^n)}\leq 1\}$$
 is precompact.

 We now use Lemma \ref{lemma, criterion of precompact} to show that $A(\calO(\calT_{K^{\d}}^b))$ is
 precompact. Firstly, note that \eqref{pf 1,1} also yields the $L^p_{\om}(\rn)$-boundedness of $\calO(\calT_{K^{\d}}^b)$.
 Then we see that $A(\calO(\calT_{K^{\d}}^b))$ is a bounded set in $L^p_{\om}(\rn)$, and
  (a) of Lemma \ref{lemma, criterion of precompact} is true.

To show (b), without loss of generality, we assume that $b$ is supported in a cube $Q$ centered at the origin.
For $x\in (2Q)^c$, by \eqref{conditon of kerenel 1} for $K^{\d}$, the H\"older inequality and $\|f\|_{L^p_{\om}(\rn)}\le1$, we have
\begin{align*}
  |\calO(\calT_{K^{\d}}^b)(f)(x)|
  = &
  \left(\sum_{i=1}^{\infty}\sup_{t_{i+1}\leq \ep_{i+1}<\ep_i\leq t_i}\left|\int_{\ep_{i+1}<|x-y|\leq \ep_i}b(y)K^{\d}(x,y)f(y)dy\right|^2\right)^{1/2}
    \\
    \lesssim &
  \left(\sum_{i=1}^{\infty}\sup_{t_{i+1}\leq \ep_{i+1}<\ep_i\leq t_i}\left|\int_{\ep_{i+1}<|x-y|\leq \ep_i}\frac{|f(y)|\chi_Q(y)}{|x-y|^n}dy\right|^2\right)^{1/2}
  \\
  \leq &
  \int_{Q}\frac{|f(y)|}{|x-y|^n}dy
  \lesssim
  \frac{1}{|x|^{n}}\int_{Q}|f(y)|dy
  \lesssim
  \frac{1}{|x|^{n}}\left(\int_{Q}\om^{-p'/p}(y)dy\right)^{1/p'}.
\end{align*}
Taking $N>2$, we then have
\begin{align*}
  \left(\int_{(2^NQ)^c}|\calO(\calT_{K^{\d}}^b)(f)(x)|^p\om(x)dx\right)^{1/p}
  \lesssim
  \left(\int_{(2^NQ)^c}\frac{\om(x)}{|x|^{pn}}dx\right)^{1/p}\left(\int_{Q}\om^{-p'/p}(y)dy\right)^{1/p'},
\end{align*}
which tends to zero as $N$ tends to infinity, where we use (iv) in Lemma \ref{lemma, property of weight}.
Thus, Lemma \ref{lemma, criterion of precompact} (b) holds.

It remains to prove that $A(\calO(\calT_{K^{\d}}^b))$ satisfies Lemma (c). Taking $z\in \bbR^n$ with $|z|\leq \frac{\d}{8}$, then
\begin{align*}
  &|\calO(\calT_{K^{\d}}^b)(f)(x+z)-\calO(\calT_{K^{\d}}^b)(f)(x)|
  \\
  &\quad\leq
  \Bigg(\sum_{i=1}^{\infty}\sup_{t_{i+1}\leq \ep_{i+1}<\ep_i\leq t_i}
  |T_{K^{\d},\ep_{i+1}}^bf(x+z)-T_{K^{\d},\ep_{i}}^bf(x+z)\\
  &\quad\quad-(T_{K^{\d},\ep_{i+1}}^bf(x)-T_{K^{\d},\ep_{i}}^bf(x))|^2\Bigg)^{1/2}.
\end{align*}
Moreover, for each $i,$ we write
\begin{align*}
&|T_{K^{\d},\ep_{i+1}}^bf(x+z)-T_{K^{\d},\ep_{i}}^bf(x+z)-(T_{K^{\d},\ep_{i+1}}^bf(x)-T_{K^{\d},\ep_{i}}^bf(x))|
\\
&\quad=
  \Big|\int_{\ep_{i+1}<|x+z-y|\leq \ep_i}(b(x+z)-b(y))K^{\d}(x+z,y)f(y)dy
\\
&\qquad-\int_{\ep_{i+1}<|x-y|\leq \ep_i}(b(x)-b(y))K^{\d}(x,y)f(y)dy\Big|
  \\
&\quad\leq
\Big|\int_{\ep_{i+1}<|x+z-y|\leq \ep_i}(b(x+z)-b(y))K^{\d}(x+z,y)f(y)dy
  \\
&\qquad  -\int_{\ep_{i+1}<|x-y|\leq \ep_i}(b(x+z)-b(y))K^{\d}(x+z,y)f(y)dy\Big|
  \\&\qquad  +  \left|\int_{\ep_{i+1}<|x-y|\leq \ep_i}(b(x+z)-b(x))K^{\d}(x,y)f(y)dy\right|
  \\
&\qquad  +
  \left|\int_{\ep_{i+1}<|x-y|\leq \ep_i}(b(x+z)-b(y))(K^{\d}(x+z,y)-K^{\d}(x,y))f(y)dy\right|
  \\
&\quad  = : I_1(i)+I_2(i)+I_3(i).
\end{align*}

We first estimate $I_2(i)$.
Since $b\in C_c^{\infty}(\bbR^n)$, we have
\begin{align*}
I_2(i)\leq
|z|\|\nabla b\|_{L^{\infty}(\rn)}\left|\int_{\ep_{i+1}<|x-y|\leq \ep_i}K^{\d}(x,y)f(y)dy\right|
\end{align*}
which yields that
\be
\begin{split}
  \left(\sum_{i=1}^{\infty}\sup_{t_{i+1}\leq \ep_{i+1}<\ep_i\leq t_i}
  \left|\int_{\ep_{i+1}<|x-y|\leq \ep_i}(b(x+z)-b(x))K^{\d}(x,y)f(y)dy\right|^2\right)^{1/2}
  \lesssim |z|\calO(\calT_{K^{\d}})(f)(x).
\end{split}
\ee
Furthermore, assume that there exists $i_0\in\mathbb N:=\{1,2,\,\cdots\}$ such that
$t_{i_0+1}<\delta\le t_{i_0}$. Then we see that for a. e. $x$,
\begin{align*}
  \calO(\calT_{K^{\d}})(f)(x)
\leq &
  \left(\sum_{i=1}^{i_0-1}\sup_{t_{i+1}\leq \ep_{i+1}<\ep_i\leq t_i}
  \left|\int_{\ep_{i+1}<|x-y|\leq \ep_i}K^{\d}(x,y)f(y)dy\right|^2\right)^{1/2}
  \\
 & +
  \sup_{t_{i_0+1}\leq \ep_{i_0+1}<\ep_{i_0}\leq t_{i_0}}
  \left|\int_{\ep_{i_0+1}<|x-y|\leq \ep_{i_0}}K^{\d}(x,y)f(y)dy\right|\\
&+
  \left(\sum_{i=i_0+1}^\infty\sup_{t_{i+1}\leq \ep_{i+1}<\ep_i\leq t_i}
  \left|\int_{\ep_{i+1}<|x-y|\leq \ep_i}K^{\d}(x,y)f(y)dy\right|^2\right)^{1/2}.
  \end{align*}
  Observe that for a. e. $x$,
\begin{align*}
&\left(\sum_{i=1}^{i_0-1}\sup_{t_{i+1}\leq \ep_{i+1}<\ep_i\leq t_i}
  \left|\int_{\ep_{i+1}<|x-y|\leq \ep_i}K^{\d}(x,y)f(y)dy\right|^2\right)^{1/2}\\
& \quad=
  \left(\sum_{i=1}^{i_0-1}\sup_{t_{i+1}\leq \ep_{i+1}<\ep_i\leq t_i}
   \left|\int_{\ep_{i+1}<|x-y|\leq \ep_i}K(x,y)f(y)dy\right|^2\right)^{1/2}\le \calO(\calT_K)(f)(x),
  \end{align*}
  and
\begin{align*} &
  \left(\sum_{i=i_0+1}^\infty\sup_{t_{i+1}\leq \ep_{i+1}<\ep_i\leq t_i}
  \left|\int_{\ep_{i+1}<|x-y|\leq \ep_i}K^{\d}(x,y)f(y)dy\right|^2\right)^{1/2}
  \\
&\quad \leq\sum_{i=i_0+1}^\infty\int_{t_{i+1}<|x-y|\leq  t_i}|K^{\d}(x,y)||f(y)|dy  \\
&\quad\leq \int_{\d/2\leq |x-y|\leq \d}|K^{\d}(x,y)||f(y)|dy\lesssim M(f)(x).
 \end{align*}
Moreover, take $\tilde \ep_{i_0+1}, \tilde \ep_{i_0}\in [t_{i_0+1}, t_{i_0}]$ such that
   \begin{align*}
   \sup_{t_{i_0+1}\leq \ep_{i_0+1}<\ep_{i_0}\leq t_{i_0}}
  \left|\int_{\ep_{i_0+1}<|x-y|\leq \ep_{i_0}}K^{\d}(x,y)f(y)dy\right|
  \le 2 \left|\int_{\tilde\ep_{i_0+1}<|x-y|\leq \tilde\ep_{i_0}}K^{\d}(x,y)f(y)dy\right|.
   \end{align*}
   We then have
     \begin{align*}
   &\sup_{t_{i_0+1}\leq \ep_{i_0+1}<\ep_{i_0}\leq t_{i_0}}
  \left|\int_{\ep_{i_0+1}<|x-y|\leq \ep_{i_0}}K^{\d}(x,y)f(y)dy\right|\\
  &\quad\le 2 \left|\int_{\tilde\ep_{i_0+1}<|x-y|\leq \tilde\ep_{i_0}}K^{\d}(x,y)f(y)dy\right|\\
  &\quad\lesssim  \left|\int_{\tilde\ep_{i_0+1}<|x-y|\leq \d}K^{\d}(x,y)f(y)dy\right|
  + \left|\int_{\d<|x-y|\leq \tilde\ep_{i_0}}K(x,y)f(y)dy\right|\\
  &\quad \lesssim M(f)(x)+\calO(\calT_K)(f)(x).
   \end{align*}
   Therefore, we conclude that for a. e. $x$,
  \begin{align*}
  \calO(\calT_{K^{\d}})(f)(x)\lesssim& \calO(\calT_K)(f)(x)+M(f)(x),
\end{align*}
where the implicit constant is independent of $i_0$, $\d$, $f$ and $x$.
From the above two estimates, the boundedness of $\calO(\calT_K)(f)$ and $M(f)$ on $L^p_{\om}(\rn)$ and
$\|f\|_{L^p_{\om}(\rn)}\le1$, we get
\begin{eqnarray}\label{pf 1,4}
   &&\left\|\left(\sum_{i=1}^{\infty}\sup_{t_{i+1}\leq \ep_{i+1}<\ep_i\leq t_i}
  I_2(i)^2\right)^{1/2}\right\|_{L^p_{\om}(\rn)}\nonumber\\
  &&\quad  \lesssim
  |z|\left(\|\calO(\calT_K)(f)\|_{L^p_{\om}(\rn)}+\|M(f)\|_{L^p_{\om}(\rn)}\right)\lesssim |z|.
\end{eqnarray}

Next, we turn to the estimate of $I_3(i)$.
Observe that $K^{\d}(x+z,y)$ and $K^{\d}(x,y)$ vanish when $|x-y|\leq \frac{\d}{4}$.
Then by \eqref{conditon of kerenel 2} for $K^{\d}$
\be
\begin{split}
  I_3(i)
  = &
  \left|\int_{\ep_{i+1}<|x-y|\leq \ep_i}(b(x+z)-b(y))(K^{\d}(x+z,y)-K^{\d}(x,y))f(y)dy\right|\\
  \lesssim &
  \int_{\ep_{i+1}<|x-y|\leq \ep_i}\frac{|z|^{\g}}{|x-y|^{n+\g}}\chi_{\{{|x-y|>\d/4}\}}(y)|f(y)|dy,
\end{split}
\ee
where $\g$ is as in \eqref{conditon of kerenel 2}.

From this, we further have
\begin{align*}
  &\left(\sum_{i=1}^{\infty}\sup_{t_{i+1}\leq \ep_{i+1}<\ep_i\leq t_i}
  \left|\int_{\ep_{i+1}<|x-y|\leq \ep_i}(b(x+z)-b(y))(K^{\d}(x+z,y)-K^{\d}(x,y))f(y)dy\right|^2\right)^{1/2}
  \\
  &\quad\lesssim
  \left(\sum_{i=1}^{\infty}\sup_{t_{i+1}\leq \ep_{i+1}<\ep_i\leq t_i}
  \left|\int_{\ep_{i+1}<|x-y|\leq \ep_i}\frac{|z|^{\g}}{|x-y|^{n+\g}}\chi_{\{{|x-y|>\d/4}\}}(y)|f(y)|dy\right|^2\right)^{1/2}
  \\
  &\quad\leq
  \sum_{i=1}^{\infty}
  \int_{t_{i+1}<|x-y|\leq t_i}\frac{|z|^{\g}}{|x-y|^{n+\g}}\chi_{\{{|x-y|>\d/4}\}}(y)|f(y)|dy
  \\
  &\quad\leq
  \int_{|x-y|>\d/4}\frac{|z|^{\g}}{|x-y|^{n+\g}}|f(y)|dy\lesssim \frac{|z|^{\g}}{\d^\g}M(f)(x),
\end{align*}
where the implicit constant is independent of $f$, $x$, $\d$ and $z$.
Thus,
\ben\label{pf 1,5}
\begin{split}
  \left\|\left(\sum_{i=1}^{\infty}\sup_{t_{i+1}\leq \ep_{i+1}<\ep_i\leq t_i}
  I_2(i)^2\right)^{1/2}\right\|_{L^p_{\om}(\rn)}
  \lesssim
  \frac{|z|^{\g}}{\d^\g}\|M(f)\|_{L^p_{\om}(\rn)}\lesssim \frac{|z|^{\g}}{\d^\g}.
\end{split}
\een

Finally, we proceed to the estimate of $I_1(i)$.
Note that
$$\chi_{\{\epsilon_{i+1}<|x+z-y|\le \epsilon_i\}}(y)-\chi_{\{\epsilon_{i+1}<|x-y|\le \epsilon_i\}}(y)\not=0$$
if and only if at least one of the following four statements holds:
\begin{itemize}
 \item [{\rm  (i)}]$\epsilon_{i+1}<|x+z-y|\le\epsilon_i$ and $|x-y|\le \epsilon_{i+1}$;
  \item [{\rm (ii)}]$\epsilon_{i+1}<|x+z-y|\le\epsilon_i$  and $|x-y|> \epsilon_i$;
  \item[{\rm (iii)}] $\epsilon_{i+1}<|x-y|\le\epsilon_i$ and $|x+z-y|\le \epsilon_{i+1}$;
  \item [{\rm (iv)}]$\epsilon_{i+1}<|x-y|\le\epsilon_i$ and $|x+z-y|> \epsilon_i$.
\end{itemize}
This further implies the following four cases:
\begin{itemize}
 \item [{\rm  (i')}]$\epsilon_{i+1}<|x+z-y|\le \epsilon_{i+1}+|z|$;
  \item [{\rm (ii')}]$ \epsilon_i<|x-y|\le \epsilon_i+|z|$;
  \item[{\rm (iii')}] $\epsilon_{i+1}<|x-y|\le \epsilon_{i+1}+|z|$;
  \item [{\rm (iv')}]$\epsilon_i<|x+z-y|\le \epsilon_i+|z|$.
\end{itemize}
We then have that
\begin{align*}
I_1(i)&\ls \int_{\mathbb R^n}|K^\d(x+z, y)|\chi_{\{\epsilon_{i+1}<|x+z-y|\le\epsilon_i\}}(y)\chi_{\{|x-y|\le \epsilon_{i+1}\}}(y)\left|f(y)\right|\,dy\\
&\quad+\int_{\mathbb R^n}|K^\d(x+z, y)|\chi_{\{\epsilon_{i+1}<|x+z-y|\le\epsilon_i\}}(y)\chi_{\{|x-y|> \epsilon_i\}}(y)\left|f(y)\right|\,dy\\
&\quad+\int_{\mathbb R^n}|K^\d(x+z, y)|\chi_{\{\epsilon_{i+1}<|x-y|\le\epsilon_i\}}(y)\chi_{\{|x+z-y|\le \epsilon_{i+1}\}}(y)\left|f(y)\right|\,dy\\
&\quad+\int_{\mathbb R^n}|K^\d(x+z, y)|\chi_{\{\epsilon_{i+1}<|x-y|\le\epsilon_i\}}(y)\chi_{\{|x+z-y|> \epsilon_i\}}(y)\left|f(y)\right|\,dy\\
&=:\sum_{k=1}^4{\rm I}_{1,\,k,\,i}.
\end{align*}

By similarity, we only estimate
$$\left(\sum_{i=1}^{\infty}\sup_{t_{i+1}\leq \ep_{i+1}<\ep_i\leq t_i}
  {\rm I}_{1,\,1,\,i}^2\right)^{1/2}.$$

Observe that if $\epsilon_{i+1}< 2|z|$, then by the fact that $|z|\le|x-y|/2$,
 we have that ${\rm I}_{1,\,1,\,i}=0$. Thus, assume that $i_1\in \mathbb N$
 such that $t_{i_1+1}<2|z|\le t_{i_1}$. We see that for any $i\ge i_1+1$,
 ${\rm I}_{1,\,1,\,i}=0.$ Therefore,
 $$\left(\sum_{i=1}^{\infty}\sup_{t_{i+1}\leq \ep_{i+1}<\ep_i\leq t_i}
  {\rm I}_{1,\,1,\,i}^2\right)^{1/2}=\left(\sum_{i=1}^{i_1}\sup_{t_{i+1}\leq \ep_{i+1}<\ep_i\leq t_i}
  {\rm I}_{1,\,1,\,i}^2\right)^{1/2}.$$
  Moreover, for $i=i_1$, we may assume that $\epsilon_{i_1+1} \ge 2|z|$. Then for
  $i\in\{1,2,\cdots, i_1\}$, since
 $\epsilon_{i+1}\ge 2|z|$,  for $r\in(1, 2)$ such that $w\in A_{p/r}$,
 by (i') and the H\"older inequality, we have
\begin{align*}
{\rm I}_{1,\,1,\,i}&\ls
\left[\int_{\mathbb R^n}|K^\d(x+z,y)|^r\left|f(y)\right|^r
\chi_{\{\epsilon_{i+1}<|x+z-y|\le \epsilon_i\}}(y)\,dy\right]^{1/r}
\left[\left(\epsilon_{i+1}+|z|\right)^n-\epsilon_{i+1}^n\right]^{1/{r'}}\\
&\ls\left[\int_{|x+z-y|\ge\d/2 }\frac{|f(y)|^r}{|x+z-y|^{rn}}\chi_{\{\epsilon_{i+1}<|x+z-y|\le \epsilon_i\}}(y)\,dy\right]^{1/r}
\left[\epsilon_{i+1}^{n-1}|z|\right]^{1/{r'}}\\
&\ls \left[\int_{|x+z-y|\ge\d/2}\frac{|f(y)|^r}{|x+z-y|^{r+n-1}}\chi_{\{\epsilon_{i+1}<|x+z-y|\le \epsilon_i\}}(y)\,dy\right]^{1/r}
|z|^{1/{r'}}.
\end{align*}
Since $r<2$, we then conclude that
\begin{align*}
  &\left(\sum_{i=1}^{\infty}\sup_{t_{i+1}\leq \ep_{i+1}<\ep_i\leq t_i}
  I_{1,1,i}^2\right)^{1/2}\\
&\quad\ls\left(\sum_{i=1}^{\infty}\sup_{t_{i+1}\leq \ep_{i+1}<\ep_i\leq t_i}
\left[\int_{|x+z-y|\ge\d/2}\frac{|f(y)|^r}{|x+z-y|^{r+n-1}}\chi_{\{\epsilon_{i+1}<|x+z-y|\le \epsilon_i\}}(y)\,dy\right]^{2/r}\right)^{1/2}|z|^{1/{r'}}\\
&\quad\ls\left(\sum_{i=1}^{\infty}
\left[\int_{|x+z-y|\ge\d/2}\frac{|f(y)|^r}{|x+z-y|^{r+n-1}}\chi_{\{t_{i+1}<|x+z-y|\le t_i\}}(y)\,dy\right]^{2/r}\right)^{1/2}|z|^{1/{r'}}\\
&\quad\ls\left(\sum_{i=1}^{\infty}
\int_{|x+z-y|\ge\d/2}\frac{|f(y)|^r}{|x+z-y|^{r+n-1}}\chi_{\{t_{i+1}<|x+z-y|\le t_i\}}(y)\,dy\right)^{1/r}|z|^{1/{r'}}\\
&\quad\ls\left(\int_{|x-y|\ge3\d/8}\frac{|f(y)|^r}{|x-y|^{r+n-1}}\,dy\right)^{1/r}|z|^{1/{r'}}\\
&\quad\ls \frac{|z|^{1/r'}}{\d^{1/r'}}\left[M\left(|f|^r\right)(x)\right]^{1/r}.
\end{align*}

Now since $w\in A_{p/r}$, by the boundedness of $M$, we see that
\begin{align}\label{pf 1,6}
&\left\|\left(\sum_{i=1}^{\infty}\sup_{t_{i+1}\leq \ep_{i+1}<\ep_i\leq t_i}
  {\rm I}_{1,\,1,\,i}^2\right)^{1/2}\right\|_{L^p_{\om}(\rn)}
  \ls  \frac{|z|^{1/r'}}{\d^{1/r'}}\left\|\left[M\left(|f|^r\right)\right]^{1/r}\right\|_{L^p_{\om}(\rn)}
  \ls \frac{|z|^{1/r'}}{\d^{1/r'}}.
\end{align}

  The equicontinuity of $A(\calO(\calT_{K^{\d}}^b))$ follows from
the combination of \eqref{pf 1,4}-\eqref{pf 1,6}.
We have now completed this proof.
\end{proof}

\begin{proof}[Proof of (2) in Theorem \ref{theorem, compactness}]
Assume that $\calV_{\r}(\calT_{K})$ is bounded on $L^p_{\om}(\rn)$ and $b\in \cmo$.
Take $K^{\d}(x,y)$ with $\d>0$ as in the proof of (1).
Arguing as in \eqref{e-osci commu donimate HL maxi}, we also have that for any $f\in L^p_{\om}(\rn)$,
\begin{equation*}
\|\calV_{\r}(\calT_{K^\d}^b)(f)-\calV_{\r}(\calT_{K}^b)(f)\|_{L^p_{\om}(\rn)}\lesssim \d \|M(f)\|_{L^p_{\om}(\rn)}\lesssim \d\|f\|_{L^p_{\om}(\rn)},
\end{equation*}
and obtain the boundedness of $\calV_{\r}(\calT_{K^{\d}}^b)$ via this inequality and the $L^p_{\om}(\rn)$-boundedness of
 $\calV_{\r}(\calT_{K}^b)$, which follows from Lemma \ref{lemma, bd} and the $L^p_{\om}(\rn)$-boundedness of $\calV_{\r}(\calT_{K})$. Moreover, by a similar argument, to show $\calV_{\r}(\calT_{K}^b)$ is compact,
we only need to verify the set
$$A(\calV_{\r}(\calT_{K^{\d}}^b)):=\{\calV_{\r}(\calT_{K^{\d}}^b)(f): \|f\|_{L^p_{\om}(\rn)}\leq 1\}$$  is precompact,
where  $b\in C_c^{\infty}(\bbR^n)$.

Without loss of generality, we assume that $b$ is supported in a cube $Q$ centered at the origin.
By the boundedness of $\calV_{\r}(\calT_{K^{\d}}^b)$, $A(\calV_{\r}(\calT_{K^{\d}}^b))$ is a bounded set on $L^p_{\om}(\rn)$.
Therefore condition (a)
of Lemma \ref{lemma, criterion of precompact} holds.

Again, by the same argument as in the proof of (1) in Theorem \ref{theorem, compactness}, we have that for any $\|f\|_{L^p_{\om}(\rn)}\leq 1$
and $x\notin 2^NQ$ with $N>2$,
$$\calV_{\r}(\calT_{K^{\d}}^b)(f)(x) \lesssim
  \frac{1}{|x|^{n}}\left(\int_{Q}\om^{1-p'}(y)dy\right)^{1/p'}.$$
And hence,
\be
\begin{split}
  \left(\int_{(2^NQ)^c}|\calV_{\r}(\calT_{K^{\d}}^b)(f)(x)|^p\om(x)dx\right)^{1/p}
  \lesssim
  \left(\int_{(2^NQ)^c}\frac{\om(x)}{|x|^{pn}}dx\right)^{1/p}\left(\int_{Q}\om^{1-p'}(x)dx\right)^{1/p'}
\end{split}
\ee
tends to zero as $N$ tends to infinity.
This proves condition (b) of Lemma \ref{lemma, criterion of precompact}.

It remains to prove that $A(\calV_{\r}(\calT_{K^{\d}}^b))$ is uniformly equicontinuous in $L^p_{\om}(\rn)$. Take $z\in \bbR^n$ with $|z|\leq \frac{\d}{8}\leq l_Q/2$.
By the sub-linearity of $\calV_{\r}$, we have
\be
\begin{split}
  &|\calV_{\r}(\calT_{K^{\d}}^b)(f)(x+z)-\calV_{\r}(\calT_{K^{\d}}^b)(f)(x)|
  \\
  &\quad\leq
  \sup_{\ep_{i}\downarrow 0}\left(\sum_{i=1}^{\infty}
  |T_{K^{\d},\ep_{i+1}}^bf(x+z)-T_{K^{\d},\ep_{i}}^bf(x+z)-(T_{K^{\d},\ep_{i+1}}^bf(x)-T_{K^{\d},\ep_{i}}^bf(x))|^{\r}\right)^{1/\r}.
\end{split}
\ee
Write
\begin{align*}
&|T_{K^{\d},\ep_{i+1}}^bf(x+z)-T_{K^{\d},\ep_{i}}^bf(x+z)-(T_{K^{\d},\ep_{i+1}}^bf(x)-T_{K^{\d},\ep_{i}}^bf(x))|
\\
&\quad=
  \Big|\int_{\ep_{i+1}<|x+z-y|\leq \ep_i}(b(x+z)-b(y))K^{\d}(x+z,y)f(y)dy
  \\
  &\qquad-\int_{\ep_{i+1}<|x-y|\leq \ep_i}(b(x)-b(y))K^{\d}(x,y)f(y)dy\Big|
  \\
  &\quad\leq
  \Big|\int_{\ep_{i+1}<|x+z-y|\leq \ep_i}(b(x+z)-b(y))K^{\d}(x+z,y)f(y)dy
  \\
  &\quad\quad-\int_{\ep_{i+1}<|x-y|\leq \ep_i}(b(x+z)-b(y))K^{\d}(x+z,y)f(y)dy\Big| \\
  &\qquad+\left|\int_{\ep_{i+1}<|x-y|\leq \ep_i}(b(x+z)-b(x))K^{\d}(x,y)f(y)dy\right|  \\
  &\qquad+
  \left|\int_{\ep_{i+1}<|x-y|\leq \ep_i}(b(x+z)-b(y))(K^{\d}(x+z,y)-K^{\d}(x,y))f(y)dy\right|\\
  &\quad=:  J_1(i)+J_2(i)+J_3(i).
\end{align*}
Observe that $J_2(i)$ is dominated by
\be
|z|\left|\int_{\ep_{i+1}<|x-y|\leq \ep_i}K^{\d}(x,y)f(y)dy\right|,
\ee
which yields that
\be
\begin{split}
  \sup_{\ep_i\downarrow 0}\left(\sum_{i=1}^{\infty}
  J_2(i)^{\r}\right)^{1/\r}
  \lesssim |z|\calV_{\r}(\calT_{K^\d})(f)(x).
\end{split}
\ee
Furthermore,
\begin{align*}
  \calV_{\r}(\calT_{K^\d})(f)(x)
  &\lesssim
\sup_{\ep_i\downarrow 0}\left(\sum_{\epsilon_{i+1}\geq \d}\left|\int_{\ep_{i+1}<|x-y|\leq \ep_i}K^{\d}(x,y)f(y)dy\right|^{\r}\right)^{1/\r}
  \\
  &\quad +
  \sup_{\ep_i\downarrow 0}\left(\sum_{\epsilon_{i}\leq \d}\left|\int_{\ep_{i+1}<|x-y|\leq \ep_i}K^{\d}(x,y)f(y)dy\right|^{\r}\right)^{1/\r}\\
  &=
  \sup_{\ep_i\downarrow 0}\left(\sum_{\epsilon_{i+1}\geq \d}
  \left|\int_{\ep_{i+1}<|x-y|\leq \ep_i}K(x,y)f(y)dy\right|^{\r}\right)^{1/\r}
  \\
  & \quad+
  \sup_{\ep_i\downarrow 0}\left(\sum_{\epsilon_{i}\leq \d}
  \left|\int_{\ep_{i+1}<|x-y|\leq \ep_i}K^{\d}(x,y)f(y)dy\right|^{\r}\right)^{1/\r}
  \\
  &\le
  \calV_{\r}(\calT_{K})(f)(x)+\int_{\d/2\leq |x-y|\leq\d}|K^{\d}(x,y)|\cdot |f(y)|dy\\
  &\lesssim \calV_{\r}(\calT_{K})(f)(x)+M(f)(x).
\end{align*}
Here the implicit constant is independent of the choice of $\d$, $f$ and $x$.
Thus, by the boundedness of $\calV_{\r}(\calT_K)(f)$ and $M(f)$, for any $f$ with $\|f\|_{L^p_{\om}(\rn)}\le 1$,
\begin{align*}
  \left\|\sup_{\ep_i\downarrow 0}\left(\sum_{i=1}^{\infty}
  J_2(i)^{\r}\right)^{1/\r}\right\|_{L^p_{\om}(\rn)}
  \lesssim
  |z|\|\calV_{\r}(\calT_K)(f)+M(f)\|_{L^p_{\om}(\rn)}\lesssim |z|.
\end{align*}

Observe that $K^{\d}(x+z,y)$ and $K^{\d}(x,y)$ vanish when $|x-y|\leq \frac{\d}{4}$. Then by \eqref{conditon of kerenel 2},
$J_3(i)$ is dominated by
\begin{align*}
  \int_{\ep_{i+1}<|x-y|\leq \ep_i}\frac{|z|^{\g}}{|x-y|^{n+\g}}\chi_{\{|x-y|>\d/4\}}(y)|f(y)|dy\ls \left(\frac{|z|}{\d}\right)^\gamma Mf(x).
\end{align*}
By the same argument as the estimate of $I_3$ in the proof of (1) of Theorem \ref{theorem, compactness}, we get
that for any $f$ such that $\|f\|_{L^p_{\om}(\rn)}\le 1$,
\begin{align*}
  \left\|\sup_{\ep_i\downarrow 0}\left(\sum_{i=1}^{\infty}
J_3(i)^{\r}\right)^{1/\r}\right\|_{L^p_{\om}(\rn)}
  \lesssim
 \left(\frac{|z|}{\d}\right)^\gamma\|M(f)\|_{L^p_{\om}(\rn)}\lesssim \left(\frac{|z|}{\d}\right)^\gamma.
\end{align*}

Finally, we turn to the estimate of $J_1(i)$. Denote
$$E_i(x,z): =\{y\in \bbR^n: \ep_{i+1}<|x+z-y|\leq \ep_i\}.$$
Recalling that $b$ is supported in $Q$ and $|z|\leq \d/8\leq l_Q/2$, we have that for any $x\in \mathbb R^n$,
\be
b(x+z)=b(x+z)\chi_{2Q}(x).
\ee
This and the sub-linearity of $\calV_{\r}$ imply that
\begin{align*}
  & \left\|\sup_{\ep_i\downarrow 0}\left(\sum_{i=1}^{\infty}J_1(i)^{\r}\right)^{1/\r}\right\|_{L^p_{\om}(\rn)}
  \\
  &\quad=
  \left\|\sup_{\ep_i\downarrow 0}\left(\sum_{i=1}^{\infty}\left|\int_{E_i(x,z)\triangle E_i(x,0)}(b(x+z)-b(y))K^{\d}(x+z,y)f(y)dy\right|^{\r}\right)^{1/\r}\right\|_{L^p_{\om}(\rn)}
  \\
  &\quad\le
  \left\|\sup_{\ep_i\downarrow 0}\left(\sum_{i=1}^{\infty}\left|\int_{E_i(x,z)\triangle E_i(x,0)}b(x+z)K^{\d}(x+z,y)f(y)dy\right|^{\r}\right)^{1/\r}\right\|_{L^p_{\om}(2Q)}
  \\
  &\qquad +
  \left\|\sup_{\ep_i\downarrow 0}\left(\sum_{i=1}^{\infty}\left|\int_{E_i(x,z)\triangle E_i(x,0)}b(y)K^{\d}(x+z,y)f(y)dy\right|^{\r}\right)^{1/\r}\right\|_{L^p_{\om}(\rn)}
=:L_1+L_2.
\end{align*}
We start with the estimate of $L_1$.
For some large positive constant $N$, denote $f_1:=f\chi_{(2^NQ)^c}$, $f_2:=f\chi_{2^NQ}$,
$$ L_1^1(x):=\sup_{\ep_i\downarrow 0}\left(\sum_{i=1}^{\infty}\left|\int_{E_i(x,z)\triangle E_i(x,0)}|K^{\d}(x+z,y)f_1(y)|dy\right|^{\r}\right)^{1/\r}$$
and
$$L_1^2(x):=\sup_{\ep_i\downarrow 0}\left(\sum_{i=1}^{\infty}\left|\int_{E_i(x,z)\triangle E_i(x,0)}|K^{\d}(x+z,y)f_2(y)|dy\right|^{\r}\right)^{1/\r}.$$
Then we write
\begin{align*}
L_1\leq &
\|b\|_{L^{\infty}(\mathbb R^n)}\left\|\sup_{\ep_i\downarrow 0}\left(\sum_{i=1}^{\infty}\left|\int_{E_i(x,z)\triangle E_i(x,0)}K^{\d}(x+z,y)f(y)dy\right|^{\r}\right)^{1/\r}\right\|_{L^p_{\om}(2Q)}
\\
\ls& \|L_1^1\|_{L^p_{\om}(2Q)}+\|L_1^2\|_{L^p_{\om}(2Q)}.
\end{align*}

Recall $|K^{\d}(x+z,y)|\lesssim \frac{1}{|x+z-y|^n}\sim \frac{1}{|y|^n}$ for $x\in 2Q$ and $y\in (2^NQ)^c$.
For $x\in 2Q$, we have
\begin{align*}
 L_1^2(x) \lesssim &
  \sup_{\ep_i\downarrow 0}\left(\sum_{i=1}^{\infty}\left|\int_{E_i(x,z)\triangle E_i(x,0)}\frac{|f_2(y)|}{|y|^n}dy\right|^{\r}\right)^{1/\r}
  \\
  \leq &
  \sup_{\ep_i\downarrow 0}\sum_{i=1}^{\infty}\int_{E_i(x,z)\triangle E_i(x,0)}\frac{|f_2(y)|}{|y|^n}dy
  \lesssim
  \int_{\bbR^n}\frac{|f_2(y)|}{|y|^n}dy=\int_{(2^NQ)^c}\frac{|f(y)|}{|y|^n}dy
  \\
  \leq &
  \left(\int_{\bbR^n}|f(y)|^p\om(y)dy\right)^{1/p}\left(\int_{(2^NQ)^c}\frac{\om^{1-p'}(y)}{|y|^{np'}}dy\right)^{1/p'}\leq \b^{(1)}_N,
\end{align*}
where $\b^{(1)}_N\rightarrow 0$ as $N\rightarrow \infty$ by (iv) in Lemma \ref{lemma, property of weight}.

Next, recall $|K_{\d}(x+z,y)|\lesssim \frac1{\d^n}$.
For fixed $N>0$, $x\in 2Q$, we have
\begin{align*}
  L_1^1(x)&\lesssim\frac1{\d^n}
  \sup_{\ep_i\downarrow 0}\left|\int_{E_i(x,z)\triangle E_i(x,0)}|f_1(y)|dy\right|^{1-1/\r}\left(\sum_{i=1}^{\infty}\left|\int_{E_i(x,z)\triangle E_i(x,0)}|f_1(y)|dy\right|\right)^{1/\r}
  \\
  &\leq\frac1{\d^n}
  \sup_{\ep_i\downarrow 0} \sup_{i}\left|\int_{E_i(x,z)\triangle E_i(x,0)}|f_1(y)|dy\right|^{1-1/\r}
  \cdot \left(\int_{\bbR^n}|f_1(y)|dy\right)^{1/\r}
  \\
  &\lesssim\frac1{\d^n}
  \sup_{\ep_i\downarrow 0}\sup_{i}\left|\int_{E_i(x,z)\triangle E_i(x,0)}\om(y)^{1-p'}\chi_{2^NQ}(y)dy\right|^{\frac{1}{p'}-\frac{1}{\r p'}}\\
  &\quad\times  \left|\int_{E_i(x,z)\triangle E_i(x,0)}|f_1(y)|^p\om(y)dy\right|^{\frac{1}{p}-\frac{1}{\r p}}
  \\
  &\lesssim\frac1{\d^n}
  \sup_{\ep_i\downarrow 0}\sup_{i}\left|\int_{E_i(x,z)\triangle E_i(x,0)}\om(y)^{1-p'}\chi_{2^NQ}(y)dy\right|^{\frac{1}{p'}-\frac{1}{\r p'}},
\end{align*}
where in the last-to-second inequality, we use the fact that for any $f$ with $\|f\|_{L^p_{\om}(\rn)}\le1$,
$$\int_{\bbR^n}|f_1(y)|dy\le \left[\int_{2Q}|f(y)|^p \omega(y) dy\right]^\frac1p\left[\int_{2Q}\omega(y)^{-\frac{p'}p} dy\right]^\frac1{p'}\ls1
$$
and the constant depends on $p$, $Q$ and $\omega$.

For the last term, we claim that,
\be
|2^NQ \cap (E_i(x,z)\triangle E_i(x,0))|\ls\b^{(2)}_{N,|z|}
\ee
uniformly for all $x\in 2Q$, $\{\epsilon_i\}$ and $i\in \mathbb N$, where $\b^{(2)}_{N,|z|}\rightarrow 0$ as $|z|\rightarrow 0$ for any fixed $N$.

In fact, for any $\{\epsilon_i\}$, assume $i_0\in\mathbb N$ be such that
$$\ep_{i_0+1}< \sqrt{n}(2+2^N)l_Q+\d\le\ep_{i_0}.$$
Then we see that for any $i$ such that $i\le i_0-1$,
$$E_i(x,z)\cap 2^NQ=\emptyset, \quad E_i(x,0)\cap 2^NQ=\emptyset.$$
Moreover, we may further assume that
$\ep_{i_0}\leq \sqrt{n}(2+2^N)l_Q+\d$.
Then for all $i\ge i_0$, by the fact that
$$(E_i(x,z)\triangle E_i(x,0))\subset (B_{x,\ep_i}\triangle B_{x+z,\ep_i})\cup (B_{x,\ep_{i+1}}\triangle B_{x+z,\ep_{i+1}}),$$
we get
\be
\begin{split}
|E_i(x,z)\triangle E_i(x,0)|
\leq &
|B_{x,\ep_i}\triangle B_{x+z,\ep_i}|+|B_{x,\ep_{i+1}}\triangle B_{x+z,\ep_{i+1}}|
\\
\lesssim &
(\ep_i+|z|)^{n-1}|z|+(\ep_{i+1}+|z|)^{n-1}|z|\\
\lesssim& (\sqrt{n}(2+2^N)l_Q+2\d)|z|\leq C_N|z|.
\end{split}
\ee
Therefore the claim follows.

This claim and the fact $\om^{1-p'}\chi_{2^NQ}\in L^1(\rn)$ yield that
\be
  \sup_{\ep_i\downarrow 0}\sup_{i}\left|\int_{E_i(x,z)\triangle E_i(x,0)}\om(y)^{1-p'}\chi_{2^NQ}(y)dy\right|^{\frac{1}{p'}-\frac{1}{\r p'}}\rightarrow 0, \ \ \text{as}\ |z|\rightarrow 0.
\ee
Hence,
\ben\label{pf, 1}
\begin{split}
L_1^1(x)
  \lesssim \frac1{\d^n}\b^{(2)}_{N,|z|}.
\end{split}
\een

Combination of the above estimates for $L_1^1$ and $L_1^2$ yields that
\be
\begin{split}
L_1\lesssim \frac1{\d^n}\|L_1^1\|_{L^p_{\om}(2Q)}+\|L_1^2\|_{L^p_{\om}(2Q)}
   \lesssim  \b^{(1)}_{N}+\frac1{\d^n}\b^{(2)}_{N,|z|}.
\end{split}
\ee
Taking sufficient large $N$ and sufficient small $|z|$, we can make $L_1$ arbitrary small.

Now, we turn to the estimate of $L_2$.
\be
\begin{split}
L_2= &\left\|\sup_{\ep_i\downarrow 0}\left(\sum_{i=1}^{\infty}\left|\int_{E_i(x,z)\triangle E_i(x,0)}b(y)K^{\d}(x+z,y)f(y)dy\right|^{\r}\right)^{1/\r}\right\|_{L^p_{\om}(\rn)}
\\
\lesssim &
\left\|\sup_{\ep_i\downarrow 0}\left(\sum_{i=1}^{\infty}\left|\int_{E_i(x,z)\triangle E_i(x,0)}b(y)K^{\d}(x+z,y)f(y)dy\right|^{\r}\right)^{1/\r}\right\|_{L^p_{\om}((2^{\widetilde N}Q)^c)}
\\
& +\left\|\sup_{\ep_i\downarrow 0}\left(\sum_{i=1}^{\infty}\left|\int_{E_i(x,z)\triangle E_i(x,0)}b(y)K^{\d}(x+z,y)f(y)dy\right|^{\r}\right)^{1/\r}\right\|_{L^p_{\om}(2^{\widetilde N}Q)}:= L_2^1+L_2^2.
\end{split}
\ee
First, we deal with $L_2^1$.
Recall $|K^{\d}(x+z,y)|\lesssim \frac{1}{|x+z-y|^n}\sim \frac{1}{|x^n}$ for $y\in Q$, $x\in (2^{\widetilde{N}}Q)^c$.
We have
\be
\begin{split}
  L_2^1\lesssim &\left\|\sup_{\ep_i\downarrow 0}\left(\sum_{i=1}^{\infty}\left|\int_{E_i(x,z)\triangle E_i(x,0)}\chi_{Q}(y)|f(y)|dy\right|^{\r}\right)^{1/\r}\cdot \frac{1}{|x|^n}\right\|_{L^p_{\om}((2^{\widetilde N}Q)^c)}
  \\
  \lesssim &
  \left\|\sup_{\ep_i\downarrow 0}\sum_{i=1}^{\infty}\int_{E_i(x,z)\triangle E_i(x,0)}\chi_{Q}(y)|f(y)|dy\cdot \frac{1}{|x|^n}\right\|_{L^p_{\om}((2^{\widetilde N}Q)^c)}
  \\
  \lesssim &
  \int_{Q}|f(y)|dy\cdot \left\|\frac{1}{|x|^n}\right\|_{L^p_{\om}((2^{\widetilde N}Q)^c)}\lesssim \b^{(3)}_{\widetilde N},
\end{split}
\ee
where $\b^{(3)}_{\widetilde N}\rightarrow 0$ as $\widetilde N\rightarrow \infty$ by (iv) in Lemma \ref{lemma, property of weight}.
Then, for fixed $\widetilde N$, by the same technique as in the estimate of $L_1^1$,
\be
\begin{split}
  L_2^2\lesssim &\frac1{\d^n}\left\|\sup_{\ep_i\downarrow 0}\left(\sum_{i=1}^{\infty}\left|\int_{E_i(x,z)\triangle E_i(x,0)}|f(y)|\chi_Q(y)dy\right|^{\r}\right)^{1/\r}\right\|_{L^p_{\om}(2^{\widetilde N}Q)}
  \lesssim\frac1{\d^n}
  \b^{(4)}_{\widetilde N,|z|},
\end{split}
\ee
where $\b^{(4)}_{\widetilde N,|z|}\rightarrow 0$ as $|z|\rightarrow 0$ for any fixed $\widetilde N$.
Hence,
\be
\begin{split}
  & \left\|\sup_{\ep_i\downarrow 0}\left(\sum_{i=1}^{\infty}\left|\int_{E_i(x,z)\triangle E_i(x,0)}(b(x+z)-b(y))K^{\d}(x+z,y)f(y)dy\right|^{\r}\right)^{1/\r}\right\|_{L^p_{\om}(\rn)}
  \\
  &\quad\lesssim
  L_1+L_2
  \lesssim
  \b^{(1)}_N+\frac1{\d^n}\b^{(2)}_{N,|z|}+\b^{(3)}_{\widetilde N}+\frac1{\d^n}\b^{(4)}_{\widetilde N,|z|}.
\end{split}
\ee
Therefore, we conclude that the set $A(\calV_{\r}(\calT_{K^{\d}}^b))$ is uniformly equicontinuous in $L^p_{\om}(\rn)$
and hence Lemma \ref{lemma, criterion of precompact} (c) holds, which finishes the proof of Theorem \ref{theorem, compactness} (2).
\end{proof}

\section{Necessity of compact oscillation and variation of commutators}
This section is devoted to the proof of Theorem \ref{theorem, necessity of cp}.
We follow the approach of \cite{GuoWuYang17Arxiv}.
For any measurable function $f$,
let $f^*$ be the non-increasing rearrangement of $f$, namely,
for any $t\in(0,\infty)$,
$$f^*(t):=\inf\left\{\alpha\in(0,\infty):\,\,|\{x\in\mathbb R^n:|f(x)|>\alpha\}|<t\right\}.$$
Recall the John-Str\"{o}mberg equivalence (see \cite{John65} and \cite{Stromberg79IUMJ}) of a function $f\in\rm BMO(\mathbb{R}^n)$
\ben\label{e-J-S equiv}
\|f\|_{\bmo}\sim \|f\|_{\bmoz}:=\sup_Qa_{\lambda}(f;Q),
\een
where for $0<\lambda<1$, the local mean oscillation of $f$ over a cube $Q$ is defined by
\be
a_{\lambda}(f;Q):=\inf_{c\in \mathbb{C}}((f-c)\chi_Q)^*(\lambda|Q|).
\ee
In \cite{GuoWuYang17Arxiv}, the following equivalent characterization of $\cmo$ in terms of the local mean oscillation was established.
\begin{lemma}\label{lemma, new cmo}(\cite{GuoWuYang17Arxiv})
  Let $f\in \bmo$. Then $f\in \cmo$ if and only if the following three conditions hold:
  \bn
  \item $\lim_{r\rightarrow 0}\sup\limits_{|Q|=r}a_{\la}(f;Q)=0$,
  \item $\lim_{r\rightarrow \infty}\sup\limits_{|Q|=r}a_{\la}(f;Q)=0$,
  \item $\lim_{d\rightarrow \infty}\sup\limits_{Q\cap [-d,d]^n=\emptyset}a_{\la}(f;Q)=0$.
  \en
\end{lemma}

In order to deal with the necessity conditions for the compact oscillation and variation of commutators,
we recall following two lemmas from \cite{GuoWuYang17Arxiv}.

\begin{lemma}[lower estimates]\label{lemma, l-es}
Let $\om\in \apn$, $\lambda\in(0, 1)$ and $b$ be a real-valued measurable function.
 For a given cube $Q$, there exists a cube $P$ with the same side length of $Q$
 satisfying $|c_Q-c_P|=k_0l_Q\ (k_0>10\sqrt{n})$, and measurable sets $E\subset Q$ with $|E|=\frac{\la}{2}|Q|$,
 and $F\subset P$ with $|F|=\frac{1}{2}|Q|$, such that for $f:=(\int_{F}\om(x)dx)^{-1/p}\chi_F$,
 and any measurable set $B$ with $|B|\leq \frac{\la}{8}|Q|$,
  \begin{equation}\label{e-pointwise equivalence osci}
  \|T_{\Omega}^b(f)\|_{L_{\omega}^p(E\bs B)}\ge C{a}_{\lambda}(b;Q).
  \end{equation}
\end{lemma}

\begin{lemma}[upper estimates]\label{lemma, u-es}
Let $b\in \bmo$, $\Om\in L^{\infty}(\bbS^{n-1})$ and $\om\in \apn$.
For a given cube $Q$, denote by $F$ the set associated with $Q$ mentioned in Lemma \ref{lemma, l-es}.
Let $f:=(\int_{F}\om(x)dx)^{-1/p}\chi_F$.
Then, there exists a positive constant $\d$ such that
\be
\|T_{\Omega}^b(f)\|_{L_{\omega}^p(2^{d+1}Q\bs 2^dQ)}
\lesssim
2^{-\d dn/p}d\|b\|_{\bmo}.
\ee
for sufficient large $d$, where the implicit constant is independent of $d$ and $Q$.
\end{lemma}
~\\
\textbf{Claim A:}
Under the assumptions of Theorem \ref{theorem, necessity of cp},
Lemma \ref{lemma, l-es} is also valid if we replace $T_{\Om}^b$
by $\tcalO(\calT_{\Om}^b)$ or by $\calV_{\r}(\calT_{\Om}^b)$.
\begin{proof}[Proof of Claim A]
 Arguing as in the proof of \cite[Proposition 4.2]{GuoWuYang17Arxiv}, we see that
 for any given cube $Q$,
 the sets $P$, $E$ and $F$ exist. Moreover, for $f:=(\int_{F}\om(x)dx)^{-1/p}\chi_F$, the following function
  \be
  (b(x)-b(y))\frac{\Om(x-y)}{|x-y|^n}f(y)
  \ee
  does not change sign on $E\times F$.
  Hence, for $x\in E$,
  \begin{align}\label{pf, 3}
  &\tcalO(\calT_{\Om}^b)(f)(x)
  \nonumber\\
& \quad=
  \left(\sum_{i=1}^{\infty}\sup_{t_{i+1}\leq \ep_{i+1}<\ep_i\leq t_i}\left|\int_{\ep_{i+1}<|x-y|\leq \ep_i}(b(x)-b(y))\frac{\Om(x-y)}{|x-y|^n}f(y)dy\right|^2\right)^{1/2}+|T_{\Om,t_1}^bf(x)|
  \nonumber\\
&\quad=
  \left(\sum_{i=1}^{\infty}\left|\int_{t_{i+1}<|x-y|\leq t_i}(b(x)-b(y))\frac{\Om(x-y)}{|x-y|^n}f(y)dy\right|^2\right)^{1/2}+|T_{\Om,t_1}^bf(x)|.
  \end{align}
Observe that for any $x\in E$, $y\in F$,
  \begin{align*}
  x-y\in E-F\subset Q-P  \subset \{x\in \bbR^n: |c_Q-c_P|/2\leq |x|\leq 2|c_Q-c_P|\},
  \end{align*}
where $Q_0$ is the cube centered at origin with side length 1.
  By this fact and the assumption
  $$\sup_{i\in \mathbb{Z}}|\{j: 2^i\leq |t_j|< 2^{i+1}\}|<\infty,$$
  there are only finite terms which are non-zero in the series in \eqref{pf, 3}.
  Thus,
  \begin{align*}
  \tcalO(\calT_{\Om}^b)(f)(x)
  = &
  \left(\sum_{i=1}^{\infty}\left|\int_{t_{i+1}<|x-y|\leq t_i}(b(x)-b(y))\frac{\Om(x-y)}{|x-y|^n}f(y)dy\right|^2\right)^{1/2}+|T_{\Om,t_1}^bf(x)|
  \\
  \sim &
  \sum_{i=1}^{\infty}\left|\int_{t_{i+1}<|x-y|\leq t_i}(b(x)-b(y))\frac{\Om(x-y)}{|x-y|^n}f(y)dy\right|+|T_{\Om,t_1}^bf(x)|
  \\
  \sim &
  \left|\int_{\bbR^n}(b(x)-b(y))\frac{\Om(x-y)}{|x-y|^n}f(y)dy\right|=|T_{\Om}^bf(x)|,
  \end{align*}
  and the implicit constant depends on $\{t_i\}_i$ but not on $x$, $Q$, $P$, $E$ and $F$. Therefore, by this fact and Lemma \ref{lemma, l-es}, \eqref{e-pointwise equivalence osci} with $T_{\Om}^b$
replaced by $\tcalO(\calT_{\Om}^b)$ holds.

On the other hand, for any $x\in E$,
\be
\begin{split}
\calV_{\r}(\calT_{\Om}^b)(f)(x)
= &
\sup_{\ep_{i}\downarrow 0}\left(\sum_{i=1}^{\infty}\left|\int_{\ep_{i+1}<|x-y|\leq \ep_i}(b(x)-b(y))\frac{\Om(x-y)}{|x-y|^n}f(y)dy\right|^{\r}\right)^{1/\r}
\\
\leq&
\left|\int_{\rn}(b(x)-b(y))\frac{\Om(x-y)}{|x-y|^n}f(y)dy\right|
= T_{\Omega}^{b}(f)(x)
\\
=&
\left|\int_{(k_0-\sqrt{n})l_Q< |x-y|\leq (k_0+\sqrt{n})l_Q}(b(x)-b(y))\frac{\Om(x-y)}{|x-y|^n}f(y)dy\right|
\leq  \calV_{\r}(T_{\Omega}^{b})(f)(x).
\end{split}
\ee
It follows that $\calV_{\r}(\calT_{\Om}^b)(f)(x)=T_{\Omega}^{b}(f)(x)$ for $x\in E$, which implies that \eqref{e-pointwise equivalence osci} with $T_{\Om}^b$
replaced by $\calV_{\r}(\calT_{\Om}^b)$ holds.
\end{proof}

We further have the following corollary directly follows from Lemma \ref{lemma, l-es} and Claim A.
\begin{corollary}\label{corollary, necessity of bd}
  Let $1<p<\infty$, $b\in L_{loc}^1(\mathbb{R}^n)$ and $\om\in \apn$.
  Let $\Om$ be a measurable function on $\bbS^{n-1}$, which does not change sign and is not equivalent to zero
  on some open subset of  $\bbS^{n-1}$.
  Then,
  \bn
  \item let $\{t_j\}_{j=1}^{\infty}$ be a sequence with $\sup_{i\in \mathbb{Z}}|\{j: 2^i\leq |t_j|\leq 2^{i+1}\}|<\infty$,
  then the $L^p_{\om}(\rn)$-boundedness of
  $\tcalO(\calT_{\Om}^b)$ implies $b\in \bmo$;
  \item the $L^p_{\om}(\rn)$-boundedness of
  $\calV_{\r}(\calT_{\Om}^b)$ implies $b\in \bmo$.
  \en
\end{corollary}
Moreover, by combining Corollary \ref{corollary, necessity of bd}, \cite[Corollary 1.4]{MaTorreaXu2017SCM} and the boundedness of $T_{\Omega,t}^{b}$, we have
another corollary on the characterization of bounded $\tcalO(\calT_{\Om}^b)$ and $\calV_{\r}(\calT_{\Om}^b)$.

\begin{corollary}\label{corollary, characterization-bd}
  Let $1<p<\infty$, $b\in L_{loc}^1(\mathbb{R}^n)$  and $\om\in \apn$, $\Om\in Lip_1(\bbS^{n-1})$ and $\Omega\not\equiv0$.
  Let $\{t_j\}_{j=1}^{\infty}$ be a sequence with $\sup_{i\in \mathbb{Z}}|\{j: 2^i\leq |t_j|<2^{i+1}\}|<\infty$ in the definition of $\tcalO$.
  Then
  \bn
  \item $b\in \bmo\Longleftrightarrow \tcalO(\calT_{\Om}^b)$ is bounded on $L^p_{\om}(\rn)$;
  \item $b\in \bmo\Longleftrightarrow \calV_{\r}(\calT_{\Om}^b)$ is bounded on $L^p_{\om}(\rn)$.
  \en
\end{corollary}

\textbf{Claim B:}
Lemma \ref{lemma, u-es} is also valid if we replace $T_{\Om}^b$
by $\tcalO(\calT_{\Om}^b)$ or by $\calV_{\r}(\calT_{\Om}^b)$.
\begin{proof}[Proof of Claim B]
It follows from the definitions of $\tcalO(\calT_{\Om}^b)$ and $\calV_{\r}(\calT_{\Om}^b)$ that for any cube $Q$, $d\in\mathbb N$ large enough,
$x\in 2^{d+1}Q\bs 2^dQ$ and the function $f:=(\int_{F}\om(x)dx)^{-1/p}\chi_F$ satisfies
\be
\tcalO(\calT_{\Om}^b)(f)(x), \calV_{\r}(T)(f)(x)\leq\int_{\bbR^n}\left|(b(x)-b(y))\frac{\Om(x-y)}{|x-y|^n}f(y)\right|dy.
\ee
Then arguing as in the proof of \cite[Proposition 4.4]{GuoWuYang17Arxiv}, we have that
\be
\left\|\int_{\bbR^n}\left|(b(x)-b(y))\frac{\Om(x-y)}{|x-y|^n}f(y)\right|dy
\right\|_{L_{\omega}^p(2^{d+1}Q\bs 2^dQ)}
\lesssim
2^{-\d dn/p}d\|b\|_{\bmo},
\ee
where the implicit constant is independent of $d$ and $Q$. Then the desired conclusion follows.
\end{proof}
\begin{proof}[Proof of Theorem \ref{theorem, necessity of cp}]
Using Lemma \ref{lemma, new cmo}, Claim A and B,  the proof of Theorem \ref{theorem, necessity of cp}
is just a repetition of the proof of \cite[Theorem 1.4]{GuoWuYang17Arxiv}.
We omit the details here.
\end{proof}
\begin{proof}[Proof of Corollary \ref{corollary, characterization-cp}]
  By a similar argument as in the proof of (2) of Theorem \ref{theorem, compactness},
  one can verify that $T_{\Omega,t_1}^b$ is compact on $L^p_{\om}(\rn)$.
  Then, the sufficiency follows from \cite[Corollary 1.4]{MaTorreaXu2017SCM},
  Theorem \ref{theorem, compactness} and the compactness of $T_{\Omega,t_1}^b$.
  The necessity follows from Theorem \ref{theorem, necessity of cp}.
\end{proof}

\appendix
\section{}

In this section, we give an example of oscillation of $\calO(\calT_{\Omega}^b)$
 such that $b\notin \cmo$ and $\calO(\calT_{\Omega}^b)$ is compact on $L^p_{\om}(\rn)$.

 To begin with, take $b$ to be a smooth function on $\rn$ such that
\be
b(x):=
\begin{cases}
  0  &|x|\leq 1;
  \\
  |x|^{1/2}, &|x|\geq 2.
\end{cases}
\ee
One can check that $b\notin \bmo$ by
\be
\begin{split}
  \lim_{r\rightarrow \infty}\frac{1}{r^n}\int_{[0,r]^n}|b(y)-b_{[0,r]^n}|dy
  =\infty.
\end{split}
\ee
However, assume $\Om\in Lip(\rn)$. We find that $\calO(\calT_{\Omega}^b)$ is a compact operator on $L^p_{\om}(\rn)$.

In fact, let $\varphi$ be a smooth bump function with $0\leq \va\leq 1$, supported in the ball $%
\{\xi: |\xi|<2\}$ and be equal to 1 on the ball $\{\xi: |\xi|\leq
1\}$. For any positive number $N\in \mathbb N$ we define $\va_N(x):=\va(x/N)$.
Define
\be
\calO_N(\calT_{\Omega}^b):=\va_N\calO(\calT_{\Omega}^b).
\ee
We claim that $\{\calO_N(\calT_{\Omega}^b)\}_{N\in \mathbb{N}}$ is a sequence of compact operators on $L^p_{\om}(\rn)$.
In fact, for any fixed $N$, we have
\be
b(x)-b(y)=b(x)\va_{2N+t_1}(x)-b(y)\va_{2N+t_1}(y)=:b_{2N+t_1}(x)-b_{2N+t_1}(y)
\ee
for every $x,y$ such that $\va_N(x)\neq 0$ and $|x-y|\leq t_1$.
From this and the definition of $\calO_N(\calT_{\Omega}^b)$, we have
\begin{align*}
  \calO_N(\calT_{\Omega}^b)(f)(x)
  &=
  \va_N(x)\left(\sum_{i=1}^{\infty}\sup_{t_{i+1}\leq \ep_{i+1}<\ep_i\leq t_i}\left|\int_{\ep_{i+1}<|x-y|\leq \ep_i}(b(x)-b(y))\frac{\Om(x-y)}{|x-y|^n}f(y)dy\right|^2\right)^{1/2}
  \\
  &=
  \va_N(x)\Big(\sum_{i=1}^{\infty}\sup_{t_{i+1}\leq \ep_{i+1}<\ep_i\leq t_i}\Big|\int_{\ep_{i+1}<|x-y|\leq \ep_i}(b_{2N+t_1}(x)-b_{2N+t_1}(y))\\
  &\quad\times\frac{\Om(x-y)}{|x-y|^n}f(y)dy\Big|^2\Big)^{1/2}
  \\
  &=\va_N(x)\calO(\calT_{\Omega}^{b_{2N+t_1}})(f)(x).
\end{align*}
Observe that $b_{2N+t_1}\in C_c^{\infty}(\rn)\subset \cmo$.
Then the compactness of $\calO(\calT_{\Omega}^{b_{2N+t_1}})$ follows from Theorem \ref{theorem, compactness}.
Since $\va_N$ is a bounded operator on $L^p_{\om}(\rn)$ as a pointwise multiplier, the operator
$\calO_N(\calT_{\Omega}^b)=\va_N(x)\calO(\calT_{\Omega}^{b_{2N+t_1}})$, as the product of a bounded operator and a compact operator, is also compact on $L^p_{\om}(\rn)$.

Finally, we claim that $\calO(\calT_{\Omega}^b)$ is the limit of $\calO_N(\calT_{\Omega}^b)$ in the sense
of operator norm, as $N\rightarrow \infty$. Then the compactness of $\calO(\calT_{\Omega}^b)$ follows.

Write
\begin{align*}
  |\calO(\calT_{\Omega}^b)(f)(x)-\calO_N(\calT_{\Omega}^b)(f)(x)|
  =
  (1-\va_N(x))\calO(\calT_{\Omega}^b)(f)(x).
\end{align*}
Denote by $t_1$ the first term of the sequence $\{t_j\}_{j=1}^{\infty}$ in the definition of $\calO(\calT_{\Omega}^b)$.
By the definition of $b$ and the mean value theorem, we have
\be
(1-\va_N(x))|b(x)-b(y)|
\leq
(1-\va_N(x))\sup_{y\in B(x,t_1)}|\nabla b(y)|\cdot |x-y|
\lesssim N^{-1/2}|x-y|
\ee
for all $|x-y|\leq t_1$.

Hence,
\begin{align*}
&(1-\va_N(x))\calO(\calT_{\Omega}^b)(f)(x)
\\
&\quad\lesssim
N^{-1/2}\left(\sum_{i=1}^{\infty}\sup_{t_{i+1}\leq \ep_{i+1}<\ep_i\leq t_i}\left|\int_{\ep_{i+1}<|x-y|\leq \ep_i}|x-y|\frac{|\Om(x-y)|}{|x-y|^n}|f(y)|dy\right|^2\right)^{1/2}
\\
&\quad\lesssim
N^{-1/2}\int_{|x-y|\leq t_1}\frac{|f(y)|}{|x-y|^{n-1}}dy\lesssim N^{-1/2}M(f)(x).
\end{align*}
It follows that
\be
\|(1-\va_N)\calO(\calT_{\Omega}^b)(f)\|_{L^p_{\om}(\rn)\rightarrow L^p_{\om}(\rn)}\lesssim N^{-1/2}\|M(f)\|_{L^p_{\om}(\rn)\rightarrow L^p_{\om}(\rn)}\lesssim N^{-1/2}\rightarrow 0,
\ee
as $N\rightarrow \infty$. We have now completed this proof.

\end{document}